\documentclass[10pt,a4paper]{amsart}
\usepackage{geometry}                		
\usepackage{graphicx}				
\usepackage{color}
\usepackage{amssymb,amsmath,amsthm}
\newtheorem{thm}{Theorem}
\newtheorem{lem}[thm]{Lemma}

\newtheorem{conj}{Conjecture}

\newtheorem{question}{Question}
\numberwithin{question}{section}

\DeclareMathOperator{\lm}{lm}

\author{Ralf Fr\"oberg}
\author{Samuel Lundqvist}
\address{   Department of Mathematics,
            Stockholm University,
            S-10691, Stockholm, Sweden}
\email{ralff@math.su.se, samuel@math.su.se}

\date{}

\title{Extremal Hilbert series}

\begin{document}

\begin{abstract}
Given an ideal of forms in an algebra (polynomial ring,
tensor algebra, exterior algebra, Lie algebra, bigraded polynomial ring), we consider the 
Hilbert series of the factor ring. 
We concentrate on the minimal Hilbert series, which is achieved when the forms are generic.
In the polynomial ring we also consider the opposite case of maximal series. 
This is mainly a survey article, but we give a lot of problems and
conjectures. The only novel results concern the maximal series in the polynomial ring.
\end{abstract}

\maketitle

\section{Introduction}
The Hilbert series of a graded commutative algebra is an important invariant in commutative algebra
and algebraic geometry.  It has long been known what the Hilbert series of rings
$k[x_1,\ldots,x_n]/I$, $I$ homogeneous, $k$ a field, can be, \cite{ma}. A much harder question is:
If $I$ is generated by forms $f_1,\ldots,f_r$ of degrees $d_1,\ldots,d_r$, what can the Hilbert series be? Not
much is known about this. The question about the minimal Hilbert series is relevant also for other kinds of graded algebras, and in these cases, even less is known.

But what is known is that given a degree sequence $d_1,\ldots,d_r$, we can construct an algebra with a minimal series in the lexicographical sense, by choosing the forms $f_i$ to be \emph{generic}.
A form $f$ in a graded $k$-algebra $R$ is generic if the coefficients are algebraically independent over the prime field of $k$. By \emph{generic forms} we 
mean that the coefficients of all forms are algebraically independent. That choosing the $f_i$'s as generic forms really gives a minimal series follows from the fact that  a non-trivial relation of generic forms specializes to a relation of specific forms. In the commutative case, this is Lemma 1 in \cite{fr}.


We call an algebra which is the quotient of an ideal generated by generic forms \emph{generic}. 
We denote the Hilbert series $\sum_{i\ge0}\dim_kR_iz^i$ of a graded $k$-algebra $R$ by
$R(z)$, and we will assume that $k = \mathbb{C}$ unless otherwise stated. 

\section{Polynomial rings} \label{sec:comm}
Let $R=k[x_1,\ldots,x_n]/(f_1,\ldots,f_r)$, $f_i$ forms of degree $d_i$,
$i=1,\ldots,r$. It is shown by Fr\"oberg-L\"ofwall, \cite{fr-lo}, that there is only a finite number of Hilbert series, and that
there is an open nonempty subspace of the space of coordinates for the $f_i$'s
on which the Hilbert series is constant and minimal. 
The idea is that given $(n,d_1,\ldots,d_r)$, there is a bound for the 
Castelnuovo-Mumford regularity, and for each degree the algebras with nonminimal
dimension constitute a closed set. Since one only has to check a finite set of dimensions,
the algebras with nonminimal dimension in those degrees is closed. Thus there exists a truly
minimal series, not only minimal in the lexicographical sense.

There is a longstanding conjecture due to the first author for this minimal Hilbert series, see \cite{fr}. 

\begin{conj}\label{fr}
$R(z)=(\prod_{i=1}^r(1-z^{d_i})/(1-z)^n)_+$, where $(\sum_{i\ge0}a_iz^i)_+=\sum_{i\ge0}b_iz^i$,
where $b_i=a_i$ if $a_j\ge0$ for all $j\le i$ and $b_i=0$ otherwise,
\end{conj}
The conjecture is
proved if $r\le n$ (complete intersection case), if $n\le2$, \cite{fr}, if $n=3$ by Anick, 
\cite{an3}, and if
$r=n+1$, it follows from \cite{st}. (In fact Stanley proves something more, see below.) 
It is known that the true series is larger or equal  to the
conjectured, which gives that if one finds an algebra with the conjectured series for some 
$(n,r,d_1\ldots,d_r)$, then the conjecture is proved for these values.
There are also some further special results when all $d_i$ are equal, $d_i=d$ for all $i$.
If $(n,d,r)=(4,d,4)$, by Backelin-Oneto, \cite{ba-on},
$(n,d,r)=(4,d,r)$ for $d=4,6,8,9$ and $(n,d,r)=(5,4,r)$, by Nicklasson, \cite{ni},
$(n,2,3)$, $n\le11$ and $(n,3,r)$, $n\le8$, by Fr\"oberg-Hollman, \cite{fr-ho}. It is also 
proved by Hochster-Laksov that the formula
is correct in the first nontrivial degree $\min\{ d_i\}+1$, \cite{ho-la}. The last result has been
generalized  in two different ways to an interval $t\le d+d'$, $d'<d$ 
(where there are no Koszul relations) if all generators are of degree $d$. The first, by Aubry, is only depending on $n$. The conjectured series is correct in degrees $t\le d+d'$ if 
$$d\ge2d'(n-1)/\sqrt[n-1]{(n-1)!}-d'+(d')^2/\sqrt[n-2]{(n-2)!}+(n-1)^2/\sqrt[n-1]{(n-1)!}-n+5,$$ \cite{au}.
There is also a generalization only depending on $r$ by Migliore and Miro-Roig. The conjectured series is correct in degrees
$t\le d+d'$ if ${d+d'+2\choose d+d'}\ge r{d'+2\choose d'}$, \cite{mi-mi}. We also mention that
the conjecture is proved in many cases by Nenashev, \cite{ne}.

Here are some open questions. A positive answer to Question \ref{q:com4} would give a positive answer to Question \ref{q:com3},
which would give a positive answer to Question \ref{q:com2}. The rings in Question \ref{q:com4} are studied in \cite{c-c-g-o}.

\begin{question} \label{q:com1}
 Let $f_1,\ldots,f_r$ be generic forms, $r \geq n+ 1$.
Stanley showed that the Hilbert series of $k[x_1,\ldots,x_n]/(f_1,\ldots,f_{n+1})$ and 
$k[x_1,\ldots,x_n]/(x_1^{d_1},\ldots,x_n^{d_n},f_{n+1})$ are equal, where $\deg(f_i)=d_i$.  Does the same hold for 
$k[x_1,\ldots,x_n]/(f_1,\ldots,f_{r})$ and 
$k[x_1,\ldots,x_n]/(x_1^{d_1},\ldots,x_n^{d_n},f_{n+1},\ldots,f_{r})$ when $r \geq n + 1$? If all $d_i=d$?
\end{question}

\begin{question} \label{q:com2}
Let $g_1,\ldots,g_r$ be generic forms of degree $d>1$, in $k[x_1,\ldots,x_n]$. Is the
Hilbert series of  $k[x_1,\ldots,x_n]/(g_1^k,\ldots,g_r^k)$ the same as the Hilberts series
of $k[x_1,\ldots,x_n]/(f_1,\ldots,f_r)$, $f_i$ generic of degree $dk$? This is conjectured
in \cite{ni}, and proved in some cases. The corresponding question for $d=1$ is not true,
see Section \ref{sec:lefschetz} and particularly Question \ref{q:linearpowers}. 
\end{question}

\begin{question} \label{q:com3}
 Let ${\bf d}=(d_1,\ldots,d_k)$ be natural numbers with $d_1+\cdots+d_k=d$.
Let $g_i$ be generic of degree $d_i$ and let $h_i=\prod_{i=1}^kg_i$. Does
$k[x_1,\ldots,x_n]/(h_1,\ldots,h_r)$ have the same Hilbert series as
$k[x_1,\ldots,x_n]/(f_1,\ldots,f_r)$, $f_i$ generic of degree $d$ if ${\bf d}\ne(1,1,\ldots,1)$?
\end{question}

\begin{question} \label{q:com4}
Let ${\bf d}=(d_1,\ldots,d_k)$ be natural numbers with $d_1+\cdots+d_k=d$, and let
$g_i=l_{i1}^{d_1}\cdots l_{ik}^{d_k}$, where $l_{ij}$ are generic linear forms.
Does $k[x_1,\ldots,x_n]/(g_1,\ldots,g_k)$ have the same Hilbert series as
$k[x_1,\ldots,x_n]/(f_1,\ldots,f_r)$, $f_i$ generic of degree $d$ if ${\bf d}\ne(1,1,\ldots,1)$?
\end{question}

\begin{question} \label{q:com5}
When $r \geq n$, the expression $(\prod_{i=1}^r(1-z^{d_i})/(1-z)^n)_+$ is a polynomial in $z$. What are the coefficients? What is the degree?
\end{question}

\begin{question} \label{q:com6}
Let $k[y_1,\ldots,y_n]$ be the ring of differential operators acting on $k[x_1,\ldots,x_n]$. Let $f$ be a generic form in $k[x_1,\ldots,x_n]$. Does there exist forms $g_1,\ldots,g_r$ in $k[y_1,\ldots,y_n]$ such that $k[x_1,\ldots,x_n]/(g_1.f,\ldots,g_r.f)$ satisfies Conjecture \ref{fr}?
\end{question}

\begin{question}
Let $f_1,\ldots,f_r$ be generic forms of degree $d$ with $r>n$. When is the Hilbert series of $k[x_1,\ldots,x_n]/I^s$ equal to the Hilbert series of  $k[x_1,\ldots,x_n]/J$, $J$ generated by $\binom{s+r-1}{r-1}$ generic forms of degree $ds$? Related questions were studied in \cite{b-f-l}.
\end{question}

\subsection*{Maximal Hilbert series} \label{sec:max}
We now turn to the opposite problem. What is the $\emph{maximal}$ Hilbert series that an algebra with given degree sequence can attend? 

It is well known that the upper bound for the Hilbert series in the uniform degree case is achieved when $I$ is a lex segment ideal and follows directly from the characterization of Hilbert series due to Macaulay \cite{ma}. However, we are not aware of any results in the mixed degree case. We can give an answer in the case of two variables.

\begin{thm} \label{thm:upper}
Let $k$ be any field and let $f_1,\ldots,f_r$ be forms in $k[x,y]$ of degrees $d_1 \geq \cdots \geq d_r$ such that $(f_1, \ldots, f_r)$ is a minimally generated ideal. Then
$$k[x,y]/(f_1,\ldots,f_r)(z) \leq \frac{1 + z + \cdots + z^{d_r-1} - (z^{d_{r-1}} + z^{d_{r-2}} +  \cdots + z^{d_1})}{1-z}.$$

\end{thm}
\begin{proof}
Without loss of generality, we can assume that $(\lm(f_1),\ldots,\lm(f_r)) = (m_1,\ldots,m_r)$ is a minimally generated monomial ideal with respect to some monomial ordering. It is well known that $k[x,y]/(f_1,\ldots,f_r)(z) \leq k[x,y]/(\lm(f_1),\ldots,\lm(f_r))(z)$, so it is enough to show that $$R(z)  \leq \frac{1 + z + \cdots + z^{d_r-1} - (z^{d_{r-1}} + z^{d_{r-2}} +  \cdots + z^{d_1})}{1-z},$$
where $R  = k[x,y]/(m_1,\ldots,m_r).$

For each $i$, $i = 1, \ldots, r$, we divide the monomials $m_1,\ldots,m_r$ into three disjoint groups. Let the first group consist of $m_i$ only. Let the second group consist of the monomials having $x$-degree less than $m_i$, and let the third group consist of the monomials having $y$-degree less than $m_i$. Since $(m_1,\ldots,m_r)$ is minimally generated,  it follows that the three groups form a partition of $\{m_1,\ldots,m_r\}$ for each $i$.

Fix one $i$. The Hilbert series of $(m_i)$ is equal to $\frac{z^{d_i}}{(1-z)^2}$. Denote by $B$ and $C$ the second and the third group respectively, and write $B  = \{m_{j_1},\ldots,m_{j_s}\}, C = \{m_{k_1},\ldots, m_{k_t}\}$. In the ideal $(m_{j_h})$, there is an infinite set $I_{h} = \{ m_{j_h}, m_{j_h} \cdot x, m_{j_h} \cdot x^2, \ldots \}$  of monomials. Likewise, in the ideal $(m_{k_h})$, there is an infinite set $J_{h} = \{ m_{k_h}, m_{k_h} \cdot y, m_{k_h} \cdot y^2, \ldots \}$ of monomials. Each element in $B$ has an unique $x$-degree, and each element in $C$ has an unique $y$-degree, so the $s+t+1=r$ sets $I_{1}, I_{2}, \ldots, I_{s}, J_{1}, J_{2}, \ldots, J_{t}, \{\text{monomials in } (m_i) \}$ are disjoint. 



Thus, the Hilbert series of $(m_1,\ldots,m_r)$ is coefficientwise greater than or equal to $\max(h_i)$, where
\begin{align*}
h_i = \frac{z^{d_i}}{(1-z)^2} + \frac{ z^{d_{1}} + \cdots + z^{d_{i-1}} + z^{d_{i+1}} + \cdots + z^{d_r} }{1-z}= \\
\frac{ z^{d_1} + \cdots + z^{d_r} - (+ z^{d_{1}+1} + \cdots + z^{d_{i-1}+1} + z^{d_{i+1}+1} + \cdots + z^{d_r+1})}{(1-z)^2}.
 \end{align*}
For $i \geq 2$, we have $$h_i - h_{i-1} = \frac{-z^{d_{i-1}+1} +z^{d_{i}+1}}{(1-z)^2}.$$
Since $d_{i-1} \geq d_{i}$, the coefficients in the series $h_i - h_{i-1}$ are non-negative. 
This shows that $h_1 \leq \cdots \leq h_r$ coefficientwise, and especially, that the Hilbert series of  $(m_1,\ldots,m_r)$ is greater than or equal to  $h_r$, so 
\begin{align*}
R(z) \leq 
\frac{1}{(1-z)^2} - \left(\frac{z^{d_r}}{(1-z)^2} + \frac{z^{d_{r-1}} + z^{d_{r-2}} + \cdots + z^{d_1}}{1-z}\right)= \\
\frac{1 + z  + \cdots + z^{d_r-1}}{1-z} - \frac{z^{d_{r-1}} + z^{d_{r-2}} + \cdots + z^{d_1}}{1-z}.
\end{align*}

\end{proof}
We now give an explicit construction of an ideal which attains the maximal possible Hilbert series. 
We first remark that in Theorem \ref{thm:upper}, we actually have $d_r \geq r-1$.  Indeed, 
suppose $I = (\lm(f_1), \ldots, \lm(f_r))  \subseteq k[x,y]$ is minimally generated. Let $m = x^a y^b$ be a monomial of least degree in $I.$ By the minimality assumption, each other generator either has $x$-degree less than $a$, or $y$-degree less than $b$. Moreover, two different generators can not have the same $x$- or $y$-degree. So there are at most $a+b+1$ monomials in $I$. It follows that if $I$ is minimally generated by forms of degrees $d_1 \geq \cdots \geq d_r$, then $d_r \geq r-1$ (and if $d_r= r-1$, then $k[x,y]/I$ is artinian).
\begin{thm}
Let $d_1 \geq \cdots \geq d_r \geq r$ and let $I = (x^{d_1}, x^{d_2-1}y^1,x^{d_3-2}y^2, \ldots, x^{d_r-(r-1)}y^{r-1})$. Then
the Hilbert series of $k[x,y]/I$ is equal to 
$$\frac{1 + z + \cdots + z^{d_r-1} - (z^{d_{r-1}} + z^{d_{r-2}} +  \cdots + z^{d_1})}{1-z}.$$

\end{thm}
\begin{proof}

The Hilbert series of the ideal $(x^{d_r-(r-1)} y^{r-1})$ is equal to $\frac{z^{d_r}}{(1-z)^2}$. Next, consider the ideal 
$(x^{d_{r-1} - (r-2)}y^{r-2}, x^{d_r-(r-1)} y^{r-1})$. This ideal consists of the monomials in $(x^{d_r-(r-1)} y^{r-1})$ together with the monomials of the form $x^i y^{r-2}$ with $i \geq d_{r-1} - (r-2)$, so the Hilbert series is equal to $\frac{z^{d_r}}{(1-z)^2} + \frac{z^{d_{r-1}}}{1-z}.$ Repeating this argument, we arrive at the conclusion that the Hilbert series of 
$I$ is equal to  $$\frac{z^{d_r}}{(1-z)^2} + \frac{z^{d_{r-1}} + z^{d_{r-2}} + \cdots + z^{d_1}}{1-z}.$$

It follows that the Hilbert series of $k[x,y]/I$ equals
\begin{align*}
\frac{1}{(1-z)^2} - \left(\frac{z^{d_r}}{(1-z)^2} + \frac{z^{d_{r-1}} + z^{d_{r-2}} + \cdots + z^{d_1}}{1-z}\right)= \\
\frac{1 + z + \cdots + z^{d_r-1}}{1-z} - \frac{z^{d_{r-1}} + z^{d_{r-2}} + \cdots + z^{d_1}}{1-z}.
\end{align*}
\end{proof}

Consider now the three variable case. It is not clear how to choose the monomials in order to maximize the series. In fact, it is not even clear that there is a maximal series in the coefficientwise sense. Let 
$\mathbf{d}= (5,4,3,2)$. Natural choices for maximizing the series would be $I =  (x^5,x^3 y,x^2 z,y^2)$ or $J = (x^2,xy^2,xz^3,y^5)$, where the ideal $I$ is constructed as follows. We start with degree five. We choose $x^5$ since it is the biggest monomial with respect to Lex, and we then proceed in descending degree order, picking  $x^3 y$ since it is the biggest monomial with respect to Lex which is not in $(x^5)$, and so on. 
The monomials in $J$ are chosen in the same manner, but with respect to ascending degree order. 

It turns out that $k[x,y,z]/I(z) < k[x,y,z]/J(z)$ coefficientwise, but in fact, there is another choice that gives a minimal series, namely $K = (x^5,x^3y,xy^2,xz)$. Here we have chosen the generators with respect to reverse Lex and in descending degree order.  In \cite{de}, revlex ideals and minimal Betti numbers were studied, and we believe that there could be a connection to the problem that we consider.   

\begin{question} \label{upperq1}

Is there a largest possible Hilbert series when $n \geq 3$ in the coefficientwise sense? If so, what is it?

\end{question}


\subsection{Resolutions}
There has been some work on the resolution of ideals generated by generic forms 
by Migliore and Miro-Roig, \cite{mi-mi},
\cite{mi-mi2}, and by Pardue-Richert, \cite{pa-ri}, \cite{pa-ri2}. 
Let $R=k[x_1,\ldots,x_n]/I=S/I$, and let $\beta_{i,j}=
\dim_k({\rm Tor}^S_i(R,k))_j$. Then we have that $R(z)=\sum_{i=1}^n(-1)^i\beta_{i,j}z^j/(1-z)^n$. If $I$ is minimally
generated by generic elements $f_1,\ldots,f_r$, $\deg(f_i)=d_i$, then $\beta_{i,j}=0$ if $j-i\ge e$,
where $e$ is the Castelnuovo regularity of $R$, so $R_e\ne0$ but $R_{e+1}=0$ in the artinian case.
It is shown that the resolution agrees with the Koszul resolution of $(f_1,\ldots,f_r)$ in degrees
$j-i\le e-2$ (and also in degree $j-i=e-1$ if the first nonpositive coefficient in 
$\prod_{i=1}^r(1-z^{d_i})/(1-z)^n$ is 0). They also show that if $f_1,\ldots,f_{n+1}$ are generic
forms, $\deg(f_i)=d_i$ which minimally generate $(f_1,\ldots,f_{n+1})$ in $k[x_1,\ldots,x_n]$,
the first nonpositive coefficient is 0 if and only if $\sum_{i=1}^{n+1}d_i$ is odd.
It is natural to conjecture that the only instances when
$\beta_{i_,j}=\beta_{i_2,j}$ should come from Koszul relations (and in fact this was conjectured
in \cite{ia2}). This is shown to be false. A small counterexample from \cite{mi-mi} is the resolution
of four forms of degrees 5,5,5, and 7 in $S=k[x,y,z]$. The minimal resolution has the form
$$0\rightarrow S[-12]^4\oplus S[-11]\rightarrow  S[-10]^7\oplus S[-11]\rightarrow S[-5]^3
\oplus S[-7]\rightarrow S\rightarrow S/I\rightarrow0.$$
Here the "ghost terms" $S[-11]$ are not explained through Koszul relations and cannot be predicted
from the Hilbert series. In case all relations have the same degree, there should be no ghost terms.

If all generators have the same degree, 
there exists a ghost term in the minimal resolution if  we have $\beta_{i_1,j}\ne0$ and $\beta_{i_2,j}\ne0$ with $i_1\ne i_2$ for some $j$. If there are no ghost
terms, then the Betti numbers can be determined from the Hilbert series, and the 
resolution is called pure.

\begin{conj}\cite[Conjecture 5.8]{mi-mi2}\label{mi-mi}
Let $I\subset R = k[x_1,\ldots,x_n]$ be the ideal of $d > n$ generically chosen forms of the same degree. Then there is no redundant term in the minimal free resolution of $R/I$. Consequently, the minimal free resolution is the minimum one consistent with the Hilbert function.
\end{conj}

\begin{thm}\cite{mi-mi2}
Suppose $R$ is an almost complete intersection, $r=n+1$, generated by generic forms.
Then Conjecture \ref{mi-mi} is true
when $n=3$, when both $n$ and $d$ are even, and when $(n,d)=(4,3)$. \end{thm}

We can add

\begin{thm}
Conjecture \ref{mi-mi} is true for any $r$ when $d=2$, $n\le5$, when $d=3$, $n\le4$, and when $d=4$,
$n=3$.
\end{thm}

This is proved with computer calculations in Macaulay2, see \cite{M2}.

There is an example in \cite{pa-ri} showing that the space of generic ideals with minimal Betti
numbers is strictly smaller than the space of ideals with generic Hilbert function. Let $I$ be
an ideal generated by five generic forms of degree four in $k[x,y,z,u]$, and let
$J=(x^4,y^4,z^4,u^4,(x+y+z+u)^4)$. Then both $k[x,y,z,u]/I$ and $k[x,y,z,u]/J$ have generic series.
The resolution of $R=k[x,y,z,u]/J=S/J$ is
$$0\rightarrow S[-9]^6\rightarrow S[-8]^9\oplus S[-7]^9\rightarrow S[-6]^{16}\rightarrow S[-4]^5
\rightarrow S\rightarrow S\rightarrow S/I\rightarrow0.$$
The resolution of $S/J$ has ghost terms $S[-7]$ in homological degrees 2 and 3, and $S[-8]$ in
homological degrees 3 and 4, but there are no ghost terms in the resolution of $S/I$.
Here is another example we have come across. Let
$R=k[x_1,x_2,x_3,x_4]/(f_1,\ldots,f_8)$, $f_i$ generic of degree 3, and 
$$T=
\frac{k[x_1,x_2,x_3,x_4]}{(x_1^3,x_2^3,x_3^3,x_4^3,(x_1+x_2+x_3)^3,(x_1+x_2+x_4)^3,
(x_1+x_3+x_4)^3,(x_2+x_3+x_4)^3)}.$$ Then $R$ and $T$ has Hilbert series as in Conjecture
\ref{fr}, $R$ has minimal Betti numbers but $T$ has ghost terms $S[-5]^3$ in homological
degrees 2 and 3, and $S[-6]^6$ in homological degrees 3 and 4. We believe that this is the
first example of a series, see the questions below.

\begin{question}
 Let $f_1,\ldots,f_r$ be generic forms of degree $d$. Does $k[x_1,\ldots,x_n]/(f_1,\ldots,f_r)$
and $k[x_1,\ldots,x_n]/(x_1^d,\ldots,x_n^d,f_{n+1},\ldots,f_r)$ have the same Betti numbers?
\end{question}

\begin{question}
Does $k[x,y,z,u]/(x^d,y^d,z^d,u^d,(x+y+z+u)^d)$ have minimal Betti numbers if $d\ne4$?
\end{question}

\begin{question}
Let $R=k[x_1,x_2,x_3,x_4]/(f_1,\ldots,f_8)$, $f_i$ generic of degree $d\ge3$ and let
$T=k[x_1,x_2,x_3,x_4]/(x_1^d,x_2^d,x_3^d,x_4^d,(x_1+x_2+x_3)^d,(x_1+x_2+x_4)^d,
(x_1+x_3+x_4)^d,(x_2+x_3+x_4)^d)$. Is it true that $R$ and $T$ has Hilbert series as in
Conjecture \ref{fr}, $R$ has minimal Betti numbers, and $T$ has ghost terms
$S[-(2d-1)]^3$ in homological degrees 2 and 3, and $S[-2d]^{d+3}$ in homological degrees
3 and 4? (This is true for $d\le8$.)
\end{question}

\begin{question} 
Let $R=k[x_1,x_2,x_3,x_4]/(f_1,\ldots,f_8)$, $f_i$ generic of degree $d\ge2$ and let
$T=k[x_1,x_2,x_3,x_4]/I$, where $I$ is generated by the 8 forms 
$(x_1\pm x_2\pm x_3\pm x_4)^d$. Is $T(z)-R(z)={d\choose2}z^{2(d-1)}$?
\end{question}

\subsection{Connections to the Lefschetz properties} \label{sec:lefschetz}

A graded algebra $A$ has the maximal rank property (MRP) if for any $d$, the map
$A_i\stackrel{f\cdot}{\rightarrow}A_{i+d}$ has maximal rank, i.e. is injective or surjective for all $i$,
if $f$ is a generic form of degree $d$.

The following lemma is easily proved.

\begin{lem}\cite[Lemma 4]{fr}
$((1-z^{d_{r+1}})(\prod_{i=1}^r(1-z^{d_i})/(1-z)^n)_+)_+=(\prod_{i=1}^{r+1}(1-z^{d_i})/(1-z)^n)_+$
\end{lem}

This gives the following equivalent formulation of Conjecture \ref{fr}.

\begin{conj}\label{C}
If $f_1,\ldots,f_r$ are generic forms, then $k[x_1,\ldots,x_n]/(f_1,\ldots,f_r)$ has the MRP.
\end{conj}

An algebra $A$ has the weak Lefschetz property (WLP) if the multiplication 
$A_i\stackrel{l\cdot}{\rightarrow}A_{i+1}$ has maximal rank for all $i$ and $d$
if $l$ is a generic linear form, and it has
the strong Lefschetz property (SLP) if $A_i\stackrel{l^d\cdot}{\rightarrow}A_{i+d}$ has maximal rank
for all $i$ if $l$ is a generic linear form. If $A$ has SLP, then $A$ has the MRP by semicontinuity, and
if $A$ has the MRP, then $A$ has the WLP (set $d=1$).

The following is proved by Migliore-Miro-Roig-Nagel in \cite{mi-mi-na}
\begin{thm}
If Conjecture \ref{fr} is true for all ideals generated by general forms in $n$ variables, then all 
ideals generated by general forms in $n + 1$ variables have the WLP.
\end{thm}

It is proved in \cite{st} that monomial complete intersections has the SLP, thus the MRP.
Hence $k[x,y,z]/(x^{d_1},y^{d_2},z^{d_3},f)$, where $f$ is generic of degree $d_4$, has
the Hilbert series from Conjecture \ref{fr}, so the conjecture is proved for $r=n+1$.

One might hope that powers of generic linear forms are sufficiently generic to prove
Conjecture \ref{fr}. This is not true. There is a conjecture due to Iarrobino \cite{ia2}, see \cite{ch}, on
when this is true. 

\begin{conj}
Let $f_1, \ldots, f_r$ be generic forms of degree $d$ and let $l_1, \ldots, l_r$ be generic linear forms.
For all $(d,n,r)$ except $r = n+2, r = n+3, (n,r) = (3,7), (3,8), (4,9), (5,14)$, the algebras 
$k[x_1,\ldots,x_n]/(l_1^{d}, \ldots, l_r^{d})$  and  $k[x_1,\ldots,x_n]/(f_1,\ldots,f_{r})$ have the same Hilbert series.
\end{conj}
Some, but not all, counterexamples can be explained by the relation
between ideals generated by powers of linear forms and ideals of fat points via apolarity
using the famous list of counterexamples by Alexander-Hirschowitz to the expected Hilbert series 
of 2-fat points, \cite{al-hi}. Using \cite{al-hi} Alessandro Oneto (private communication)
has proved that the only cases when $d$'th powers of linear forms have unexpected 
linear syzygies is when
$(d,n,r)=(2,5,7),(3,3,5),(3,5,7),(3,4,9),(3,5,14)$.

\begin{question}
 Does $d$'th powers of generic linear forms have the same Hilbert series as the same 
number of generic forms of degree $d$ if $d>>0$?
\end{question}

\begin{question} \label{q:linearpowers}
Are the exceptions when $d$'th powers of generic linear forms not having the same Hilbert
series as the same number of generic forms of degree $d$ finite for each $n$? For each $d$?
\end{question}

\begin{question}\label{mo1} 
Let $f_i$ be generic forms of degree $d_i$ and let $l_i$ be generic linear forms. Does
$k[x_1,\ldots,x_n]/(l_1^{d_1}, \ldots, l_s^{d_s}, f_{s+1}, \ldots, f_{r+s})$  and  $k[x_1,\ldots,x_n]/(f_1,\ldots,f_{r+s})$ have the same Hilbert series if $s\le n$?  This is conjectured in \cite{mo}. Notice the relationsship with 
 Question \ref{q:com1}.

\end{question}

\begin{question}
 Consider the $2^{n-1}$ linear forms $l_1,\ldots,l_{2^{n-1}}$ in $k[x_1,\ldots,x_n]$ which are
the sum of an odd number of variables. Does $k[x_1,\ldots,x_n]/(l_1^d,\ldots,l_{2^{n-1}}^d)$ have
the same Hilbert series as $k[x_1,\ldots,x_n]/I$, $I$ generated by $2^{n-1}$ generic forms of
degree $d$ if $n\le6$? It is true for $n\le4$, false for $n=7$.
\end{question}

\subsection{Gr\"obner bases}
A monomial ideal $I$ is called almost degrevlex if $m\in I$ implies that $m'\in I$ for all
monomials $m'$ larger than $m$ with $\deg m'=\deg m$. The following conjecture is given
by Moreno-Socias in \cite{mo}.

\begin{conj}\label{mo} The initial monomial in degrevlex order of an ideal generated by 
generic forms is almost degrevlex. 
\end{conj}

It is proved in \cite{pa} that Conjecture \ref{mo} implies
Conjecture \ref{fr}. It is also shown in \cite{pa} that the following conjecture is equivalent
to Conjecture \ref{fr}.

\begin{conj}\label{pa} If $I=(f_1,\ldots,f_n)$ is generated by generic elements in
$k[x_1,\ldots,x_n]$, then multiplication with $x_{n-i}$ on 
$k[x_1,\ldots,x_n]/(gin(I),x_n,x_{n-1},\ldots,x_{n-i+1})$ has maximal rank for $i = 0, \ldots,n-1$.
\end{conj}

\subsection{Fewnomials}
Monomial and binomial ideals are far from being general enough, but they are easy to work with. 



\begin{question} \label{qfew1}
For which $r$ are there monomials $m_1,\ldots,m_r$ of degree $d$ making the algebra $k[x_1,\ldots,x_n]/(m_1,\ldots,m_r)$ satisfy Conjecture \ref{fr}?
\end{question}

\begin{question} \label{qfew2}
Given $d_1,\ldots,d_r$, which is the smallest and largest possible Hilbert series that
$k[x_1,\ldots,x_n]/(m_1,\ldots,m_r)$ can have, $m_i$ a monomial of degree $d_i$? 
\end{question}

\begin{question}
What is the answer if we replace monomials by binomials in Question \ref{qfew1} and \ref{qfew2}? A mix between binomials and monomials? Does it, for each $r$, exist a binary ideal such that the Hochster-Laksov result holds?
\end{question}

\section{The tensor algebra}
We will denote the tensor algebra over an $n$-dimensional $k$-vector space with 
$k\langle x_1,\ldots,x_n\rangle$ and a two-sided ideal generated by $f_1,\ldots,f_r$
with $(f_1,\ldots,f_r)$. It is shown by Govorov, \cite{gov1}, that if the $f_i$'s are monomials,
then the Hilbert series $k\langle x_1,\ldots,x_n\rangle/(f_1,\ldots,f_r)$ is a rational function,
and he conjectured, \cite{gov2}, that the same is true if the $f_i$'s are homogeneous. This was shown
to be false by Shearer, \cite{she}. It is shown in \cite{fr-lo} that even if the number of 
generators and the
number of relations and their degrees are fixed, there is in general an infinite number of Hilbert
series, and that there is no open set, even in the usual Euclidean sense with $k=\mathbb R$, 
in the space of coefficients for which the Hilbert series is constant. The smallest counterexample we know of
is for three relations of degree 2 in $k\langle x_1,x_2,x_3,x_4\rangle$.
It is shown in \cite{fr-lo} that 
$\mathbb R\langle x_1,x_2,x_3,x_4\rangle/(x_1x_2-x_1x_3,x_2x_3-x_3x_2-x_2^2/q,x_2x_4)$,
where $q\in\mathbb N_+$, has series $(1-4z+3z^2-z^{q+3})^{-1}$ and that
$\mathbb R\langle x_1,x_2,x_3,x_4\rangle/(x_1x_2-x_1x_3,x_2x_3-x_3x_2,x_2x_4)$
has series  $(1-4z+3z^2)^{-1}$. It follows that the set of algebras with three relations of 
degree two with series $(1-4z+3z^2)^{-1}$ is dense, but not open in $\mathbb R^{48}$, 
and that there is infinitely many series with three relations of degree two in four variables.

\smallskip
The remaining part of this section is taken from two papers by
Anick, \cite{an1} and \cite{an2}. We restrict to the case when $\deg(x_i)=1$ for all $i$. In Anick's
papers also $\deg(x_i)>1$ is treated. 
For two power series 
 $A(z)=\sum_{i=0}^\infty a_iz^i$ and $B(z)=\sum_{i=0}^\infty b_iz^i$ Anick defines 
 $A(z)<_{lex}B(z)$ if for some $n$ we have $a_i=b_i$ if $i<n$ and $a_n<b_n$. 
 It is shown in \cite{an2}
 that for fixed $n,d_1,\ldots,d_r$ the set $S$ of Hilbert series of algebras 
 $k\langle x_1,\ldots,x_n\rangle/(f_1,\ldots,f_r)$, $\deg(f_i)=d_i$ is well ordered. Then
 $\inf(S)$ is called the generic series, and an algebra with generic series is called generic.
 It is not clear that a generic algebra exists, and if it does
 may very well depend on the field $k$. If the field is very large
 (contains the algebraic closure of an infinite transcendtal extension of its prime field) it
 is shown that generic algebras do exist.

If $R=k\langle x_1,\ldots,x_n\rangle/I$ is a graded
quotient and $f\in R_d$ then coefficientwise, $R/(f)(z)\ge R(z)(1+z^dR(z))^{-1}$ . If there is equality
$f$ is called strongly free (or inert), \cite{an1}. 
This gives that if $f_1,\ldots,f_r$ are elements in
$k\langle x_1,\ldots,x_n\rangle$, $f_i\in k\langle x_1,\ldots,x_n\rangle_{d_i}$, then the
inequality
$k\langle x_1,\ldots,x_n\rangle/(f_1,\ldots,f_r)(z)\ge((1-nz+z^{d_1}+\cdots+z^{d_r})^{-1})_+$
holds coefficientwise. and if the Hilbert series equals $(1-nz+z^{d_1}+\cdots+z^{d_r})^{-1}$, 
$\{ f_1,\ldots,f_r\}$ is called a strongly free set. (This is shown to be
equivalent to $k\langle x_1,\ldots,x_n\rangle/(f_1,\ldots,f_r)$ being of global dimension 
$\le2$.) 
If one is able to find an algebra $k\langle x_1,\ldots,x_n\rangle/(f_1,\ldots,f_r)$, $\deg(f_i)=d_i$ with 
Hilbert series $((1-nz+z^{d_1}+\cdots+z^{d_r})^{-1})_+$, this algebra is generic. The concept strongly free is the
counterpart to regular sequences in the commutative case. It is tempting to conjecture that
 for any set of forms $g_1,\ldots,g_r$ with $g_i\in k\langle x_1,\ldots,x_n\rangle_{d_i}$, that
 $k\langle x_1,\ldots,x_n\rangle/(g_1,\ldots,g_r)(z)\ge((1-nz+z^{d_1}+\cdots+z^{d_r})^{-1})_+$,
 and that there is some set $\{ g_1,\ldots,g_r\}$ which gives equality. The inequality is
 proved in \cite{an2}. There it is also shown that if $\deg(g_1)=2$, $\deg(g_2)=d\ge7$,
 there is no algebra $k\langle x_1,x_2\rangle/(g_1,g_2)$ with series $((1-2z+z^2+z^d)^{-1})_+$,
 so the second part of the statement is not true. In case $\deg(f_i)=2$ for all $i$, one can say a
 bit more. If $r\le n^2/4$ there are strongly free sets, so in this case the generic series is
 $(1-nz+rz^2)^{-1}$ \cite[Lemma 5.10]{an2}. If $r>n^2/2$ there exists an algebra with series
 $((1-nz+rz^2)^{-1})_+$.

 We conclude this section with some open 
 questions. They all concern the case when the ideal is generated by $r$ forms of degree 2. 
 Also c.f \cite{an2}.
 
 \begin{question}
 Is the generic series equal to $(1/(1-nz+rz^2))_+$ also if $n^2/4<r<n^2/2$?
 \end{question}
 
 \begin{question}
 For which $r$ is the generic algebra artinian?
 \end{question}
 
 \begin{question}
 Does the space of coefficients for rings with generic series contain an open nonempty set?
 \end{question}
 
 \begin{question}
 Does the Hochster-Laksov result hold in the tensor algebra? For $n = 2$?
 \end{question}
 
 
  \section{Lie algebras}
 We consider $\mathbb N$-graded Lie algebras.
 The Jacobi identity is $[[a,b],c]=[a,[b,c]]+[[a,c],b]$. To get quotients of associative
 algebras we study the universal enveloping algebra of Lie algebras. Thus we study
 rings of type $T(V)/I$, where $I$ is generated by Lie elements. We
 know that $((1-nz+\sum_{i=1}^rz^{d_i})^{-1})_+$, where $n$ is the number of generators, and
 we have $r$ relations of degree $d_1,\ldots,d_r$, is a lower bound for the Hilbert series.
 Anick \cite{an2}  gives an example,
 $k\langle x,y\rangle/(f_1,f_2)$, where $\deg(f_1)=3$, $\deg(f_2)=8$. ($f_1$ can after
 a change of basis be chosen as $[x,[x,y]]=[x^2,y]$.) He shows that the series is not
 the expected $(1-2z+z^3+z^8)^{-1}$. He also shows \cite{an1} that if all relations
 are of degree two, so they are linear combinations of commutators of variables,
 then there are strongly free sets of Lie elements of size $r$ if $r\le n^2/4$, so
 the Hilbert series is $(1-nz+rz^2)^{-1}$ \cite[Lemma 5.10]{an2}.
 
 \medskip
 \noindent
 {\bf Example} If $r\le n-1$, we can get the expected Hilbert series by
 choosing commutators. If $I=([x_1,x_2],[x_1,x_3],\ldots,[x_1,x_n])$, then 
 $k\langle x_1,\ldots,x_n\rangle/I$ has
 the expected series $(1-nz+(n-1)z^2)^{-1}$ as a simple
 calculation shows.
 
 \begin{question}
 Is the generic series for $r$ generic Lie elements of degree 2 in $n$ 
 variables equal to $((1-nz+r z^2)^{-1})_+$?
 \end{question}
 
 
\begin{question}
 How many generic Lie elements of degree $d$ is required for $L$  to be nilpotent?
 \end{question}

 \section{Exterior algebras}
 Let $V$ be a $k$-vector space of dimension $n$ and denote by $\Lambda(V)$ the exterior
 algebra on $n$ generators, that is, $\Lambda(V) = k\langle x_1,\ldots,x_n\rangle/(x_ix_j+x_jx_i,1\le j\le n,x_i^2,i=1,\ldots, n)$.
 In the same way as for commutative algebras, one can show that there is a
 coefficientwise minimal series which is achieved on a non-empty Zariski open
 set of the coefficients, and that we have the corresponding inequality for the
 Hilbert series. Thus, if $I$ is an ideal generated by generic forms $f_i$, $\deg(f_i)=d_i$,
 $i=1,\ldots,r$, then $\Lambda(V)/I(t)$ is coefficientwise smaller or equal to
 $(\prod_{i=1}^r(1-z^{d_i})(1+z)^n)_+$.
 A first guess could be that if $I$ is an ideal generated by generic forms, then
 $\Lambda V/I(z)=k[x_1,\ldots,x_n]/((x_1^2,\ldots,x_n^2)+J)(z)$, where $J$ has "the
 same" generators as $I$. It is true for a principal
 ideal generated by an element of even degree \cite{mo-sn}, so if $f$ is a 
 generic element of even degree
 $d$, then the Hilbert series is $((1-z^d)(1+z)^n)_+$.  But since $f \subseteq { \rm Ann} f$ when $f$ is odd, the guess is not true even for principal ideals.

In \cite{lu-ni}, the series for $\Lambda(V)/(f)$ is determined for $d$ odd and equal to $n-2$ (the expected series + $t^{n-1}$) and $n-3$ (the expected series) and for $(9,3)$ (the expected series $+3t^6$). It is also shown that the series is not the expected one when 
the form has odd degree $d$, $n$ is odd, $3d > n$ and $\binom{n}{(n+d)/2}$ is odd . The last result gives a counterexample to Conjecture 6.1 in \cite{mo-sn}.

When the ideal is generated by two general quadratic forms, it was noticed in \cite{fr-lo2} that 
for $n=5$, the Hilbert series differs from the expected series by one in degree three. The following conjecture, 
providing a combinatorial description of the series, is given in \cite{cr-lu-ne}.

\begin{conj}\label{conj:paths}
Let $f$ and $g$ be generic forms. Then the Hilbert series of $\Lambda(V)/(f,g)$ is equal to 
$1 + a(n,1) t + a(n,2) t^2 + \cdots + a(n,s)t^s + \cdots$, where $a(n,s)$ is the number of lattice paths inside the rectangle $(n+2-2s)\times (n+2)$ from the bottom left corner to the top right corner with moves of two types; $(x,y)\rightarrow (x+1,y+1) \textrm{\ or\ } (x-1,y+1)$. 
\end{conj}
It is proved in \cite{cr-lu-ne} that the series in Conjecture \ref{conj:paths} is an upper bound for $\Lambda(V)/(f,g)$ using the theory of matrix pencils.


\begin{question}
Let $l_1$ and $l_2$ be generic linear forms in $k[x_1,\ldots,x_n]$. Does the Hilbert series of $k[x_1,\ldots,x_n]/(x_1^2,\ldots,x_n^2,l_1^2,l_2^2)$ equal that for 
$\Lambda(V)/(f,g)$, $f$ and $g$ generic quadratic forms? This is conjectured in \cite{cr-lu-ne}. The corresponding question for three linear forms is not true.
\end{question}

\begin{question}
The Hochster-Laksov result does not hold in the exterior algebra when $n = 5, d_1=d_2=2$, as showed in \cite{fr-lo2}. Is this the only counterexample? 
\end{question}

\begin{question}
Given that even the principal case is unsolved, it seems hard to predict the Hilbert series for generic ideals in $\Lambda(V)$. But the only counterexamples from the naively expected series given by Conjecture \ref{fr} that we are aware of is when $r$ is small compared to $n$. Does it hold that the Hilbert series for $\Lambda(V)/(f_1,\ldots,f_r)$ is equal to the Hilbert series for $k[x_1,\ldots,x_n]/(f_1,\ldots,f_r, x_1^2,\ldots,x_n^2)$ when $r \geq n$ ?
\end{question}

\begin{question}
Consider a supercommutative algebra, generated by even and odd elements.
 In this setting the usual commutative algebra is a superalgebra with even generators, while the exterior algebra is a superalgebra with odd generators. Given a degree sequence $(d_1,\ldots,d_r)$, what is the minimal Hilbert series such an algebra can have? If at most half of the elements are odd?
\end{question}
 

 

 

 \section{Bigraded algebras}
 The coordinate ring of $\mathbb P^m\times\mathbb P^n$ is $R=k[x_0,\ldots,x_m,y_0,\ldots,y_n]$,
 where $\deg x_i=(1,0)$ and $\deg y_i=(0,1)$. We propose the problem to determine the Hilbert
 series of generic (bigraded) ideals in $R$. Suppose the ideal is generated by $f_1,\ldots,f_r$,
 $\deg f_i=(d_i,e_i)$. A natural guess is that the bigraded Hilbert series of $R/(f_1,\ldots,f_r)$ is
 $$(\prod_{i=1}^r(1-x^{d_i}y^{e_i})/((1-x)^2(1-y)^2))_+,$$ where $(\sum_{i,j\ge0}a_{i,j}x^iy^j)_+=
 \sum b_{i,j}x^iy^j$, where $b_{i,j}=a_{i,j}$ if $a_{k,l}>0$ for all $(k,l)\le(i,j)$ and $b_{i,j}=0$
 otherwise. 
 The guess is not true in general; the first counterexample that we have found is 
for three relations of degree $(2,1)$ in the coordinate ring of $\mathbb P^1\times\mathbb P^2$.
 But we conjecture that it is true when $m=n=1.$

 \begin{conj}
  If $f_1,\ldots,f_r$ are generic bigraded forms, $\deg f_i=(d_i,e_i)$, in $R=k[x_0,x_1,y_0,y_1]$,
  then the Hilbert series of $R/(f_1,\ldots,f_r)$ is equal to $(\prod_{i=1}^r(1-x^{d_i}y^{e_i})/((1-x)^2(1-y)^2)))_+$.
 \end{conj}
 
We
 have checked the conjecture on computer in case all relations are of the same degree
 $(d,e)\le(3,3)$ and in case of all degrees are of the form $(d_i,e_i)$ with $d_i+e_i=3$,
 and the corresponding claim in a few examples in the coordinate ring of 
 $\mathbb P^1\times\mathbb P^1\times\mathbb P^1$. 
 We
 conclude with some open questions.
 
 
 \begin{question}
 If $f_1,\ldots,f_r$ are generic multigraded forms in the coordinate ring $R$ of 
 ${\mathbb P}^1\times\cdots\times{\mathbb P}^1=({\mathbb P}^1)^n$, $\deg f_i=(d_{1i},\ldots,d_{ni})$,
 is the Hilbert series of $R/(f_1,\ldots,f_r)$
 $$(\prod_{i=1}^r(1-x_1^{d_{1i}}\cdots x_n^{d_{ni}})/\prod_{i=1}^n(1-x_i)^2)_+?$$
 \end{question}
 
 \begin{question}
 What is the Hilbert series of an ideal generated by generic bigraded forms in the coordinate ring
 of $\mathbb P^m\times\mathbb P^n$?
 \end{question}

 \section{Positive characteristic}

 Conjecture \ref{fr} is true for any field when $n \leq 2$, \cite{fr}.  When $n =3$, it is 
 true when $k$ is infinite \cite{an3}. The Hochster-Laksov result \cite{ho-la}, that Conjecture \ref{fr} is 
 correct up to the first non-trivial degree, is not dependent upon the characteristics, but on the 
 field being algebraically closed.
 
 We believe that Conjecture \ref{fr} is true for any field.
 However, some of the questions that we have posed in the previous sections have a negative answer.
 
Let us for instance consider Question \ref{q:com1}. Let $r=n+1$ and consider uniform degrees of the generators. It is well known that the Hilbert function of a CI with forms of degree 
$d$ is symmetric among a line trough $n(d-1)/2$. Now consider characteristic $p\geq d$. Then
 $f_{r+1}^{p-1} \cdot f_{r+1} = f_{r+1}^p = 0$ in $k[x_1,\ldots,x_n]/(x_1^d,\ldots,x_n^d)$, so
if $(d^{p-1} + d^{p})/2 \leq n(d-1)/2$, the series is not the one given in Conjecture \ref{fr}. For $p = d = 2$, this happens 
when $n \geq 6$.

\begin{question}
Does the Hochster-Laksov result hold when $k$ is finite?
\end{question}

\begin{question}
Is Conjecture \ref{fr} true when $n = 3$ and $k$ finite?
\end{question}

\begin{question}
Stanley used an explicit choice of forms when proving Conjecture \ref{fr} for $r = n+1$. As we have seen above, 
this choice does not work in positive characteristic. Can one give a construction valid also for positive characteristic? 
\end{question}


\medskip\noindent
{\bf Acknowledgements.}
We thank the participants of the \emph{Stockholm problem solving seminar} for many fruitful discussions related to the questions and conjectures presented in this paper.


\end{document}